\UseRawInputEncoding
\documentclass[11pt,twoside,reqno]{amsart}
\usepackage{url}
\usepackage{mathrsfs}
\usepackage{latexsym}
\usepackage{amsmath}
\usepackage{amssymb}
\usepackage{amsthm}
\usepackage{amscd}
\usepackage{amsfonts}
\usepackage{epsfig}
\usepackage{multicol}
\usepackage{CJK}
\usepackage{bbm}
\usepackage{amsmath, amsthm, amsfonts, amssymb,mathtools}
\usepackage{mathrsfs}
\usepackage{enumitem} % 用于列表格式控制

\allowdisplaybreaks
   
\numberwithin{equation}{section}

\theoremstyle{plain}
\newtheorem{thm}{Theorem}[section]
\newtheorem{lem}[thm]{Lemma}
\newtheorem{prop}[thm]{Proposition}
\newtheorem{rem}[thm]{Remark}
\newtheorem{cor}[thm]{Corollary}

\newtheorem{conj}[thm]{Conjecture}
\newcommand\restr[2]{{% we make the whole thing an ordinary symbol
		\left.\kern-\nulldelimiterspace % automatically resize the bar with \right
		#1 % the function
		\right|_{#2} % this is the delimiter
}}

\newcommand{\pf}{\noindent\begin {proof}}
\newcommand{\epf}{\end{proof}}
\makeatletter

\begin{document}
	\title[Trigonometric Determinants via special values ]{Trigonometric Determinants via special values of Dirichlet $L$-Functions}

\author{Liwen Gao}
\address{Liwen Gao: Department of Mathematics, Nanjing University,
	Nanjing 210093, China;  gaoliwen1206@smail.nju.edu.cn}

\author{Xuejun Guo$^{\ast}$}
\address{Xuejun Guo$^{\ast}$: Department of Mathematics, Nanjing University,
	Nanjing 210093, China;  guoxj@nju.edu.cn}
\address{$^{\ast}$Corresponding author}

\thanks{The authors are supported by National Natural Science Foundation of China (Nos. 11971226, 12231009).}

\date{}
\maketitle

\noindent

\begin{abstract} 
	In this paper, we investigate the determinants involving some trigonometric functions. We establish a connection between these determinants and the special values of Dirichlet $L$-functions, thereby extending Guo's results [Determinants of trigonometric functions and class numbers. Linear Algebra Appl. \textbf{653} (2022), 33–43] to arbitrary positive integers $n$. In addition, we also prove a conjecture raised by Zhi-Wei Sun. Our main tool is the spectral decomposition of some linear operators. By the same method we obtain an explicit formula for the determinants of sine matrices. This formula is expressed as a product of Gauss sums attached to Dirichlet characters.
\end{abstract}
\medskip

\textbf{Keywords:} Determinants, spectral decomposition, $L$-functions, Gauss sums, Dirichlet characters.

\vskip 10pt

\textbf{2020 Mathematics Subject Classification:} 11C20,  11M20, 47A10.

	\section{Introduction}
	For positive integers $m$, $n$, $R_n(m)$ denotes the least positive residue of $m\bmod{n}$ and if $(m,n)=1$, $m'$ denotes the inverse of $m\bmod{n}$. Let $p$ be an odd prime and let  $M_p=(R_p(rs'))_{1\le r,s\le (p-1)/2}$. Then the classical Maillet determinant is defined as $D_p=\det(M_p),$ which was first introduced by Maillet in \cite{m1}.
	 
	In 1913, Malo \cite{m2} conjectured that this determinant equals $(-p)^{(p-3)/2}$.  
	Later, Carlitz and Olson \cite{co} showed that this conjecture was incorrect by proving
	\[
	D_p=\pm p^{(p-3)/2}h^-_p,
	\]
	where $h_p^-$ denotes the relative class number of $\mathbb{Q}(\zeta_p)$.  
	This shows that the Maillet determinant never vanishes.
	
	In 1984, Wang \cite{w1} extended the definition of the Maillet determinant to arbitrary positive integers $n$. Let $S_n=\{a\thinspace |\thinspace 1\le a<\frac{n}{2},(a,n)=1\}$ and $U_n=\{a\thinspace |\thinspace -\frac{n}{2}< a<\frac{n}{2},(a,n)=1\}$. The Maillet determinant was given by
	$$D_n=\det\bigg(R_n(rs')\bigg)_{r,s\in S_n}.$$
 Let $n$ be a positive integer. From now on, write $\phi(n)$ for the Euler function of $n$ and set $m=\frac{\phi(n)}{2}$. When $n$ is odd, Wang also established the following result.
		\begin{equation}\label{wangk}
		D_n=2^{1-m}\prod\limits_{\chi\, \text{odd}}\bigg(\sum\limits_{a\in U_n}a\chi(a)\bigg).
	\end{equation}
where $\prod\limits_{\chi\, \text{odd}}$ stands for the product over all odd Dirichlet characters modulo $n$.

 In \cite{s1}, Sun proved some elegant formulas for the determinants of tangent functions and proposed some conjuctures. Three of these conjectures were as follows.
	\begin{conj}\label{conj:new5.1}
		If $p$ is an odd prime, then 
		\begin{equation*}\label{eq:new5.1}
			\left(\frac{-2}{p}\right)\,
			\frac{\det\!\Big(\cot\!\bigl(\tfrac{jk\pi}{p}\bigr)\Big)_{1\le j,k\le (p-1)/2}}
			{\,2^{\tfrac{p-3}{2}}\,p^{\tfrac{p-5}{4}}}\;\in\;\mathbb{Z}_{>0},
		\end{equation*}
	where \(\big(\frac{\cdot}{p}\big)\) denotes the Legendre symbol.
	In addition, if $p\equiv 3 \pmod 4$, this integer is a multiple of $h(-p)$, where $h(-p)$ is the class number of $\mathbb{Q}(\sqrt{-p})$.
	\end{conj}

		\begin{conj}\label{conj:new5.2}
		If $n$ is a positive odd integer, then 
		\begin{equation*}\label{eq:new5.2}
			\frac{\det\!\Big(\tan\!\bigl(\tfrac{jk\pi}{n}\bigr)\Big)_{1\le j,k\le (n-1)/2}}
			{(n)^{(n-1)/4}}\;\in\;\mathbb{Z}_{>0}.
		\end{equation*}
	\end{conj}
	
	\begin{conj}\label{conjcsc}
		If $p$ is an odd prime, then 
		\begin{equation*}\label{eq:new5.3}
			c_p:=\frac{1}{2^{\tfrac{p-1}{2}}\,p^{\tfrac{p-5}{4}}}\det\!\Big(\csc\!\bigl(
			\tfrac{2jk\pi}{p}\bigr)\Big)_{1\le j,k\le (p-1)/2}\;\in\;\mathbb{Z}.
		\end{equation*}
		Furthermore, $c_p=1$ if $p\equiv3\pmod8$, and $c_p=0$ if $p\equiv7\pmod 8$.
	\end{conj}
	
 Conjecture \ref{conj:new5.1} was deeply connected with the Maillet determinant by a formula of Eisenstein and was resolved by Guo in \cite{g}, where it was shown that
 \begin{equation}\label{Guocot}
\det\!\Big(\cot\!\bigl(\tfrac{jk\pi}{p}\bigr)\Big)_{1\le j,k\le (p-1)/2}=\left(\frac{-2}{p}\right)2^{\tfrac{p-3}{2}}p^{\tfrac{p-5}{4}}h_{p}^{-}. 
\end{equation}
%As a byproduct of the proof, Guo also established the identity
%\begin{equation}\label{Guosin}
%\det\!\Big(\sin\!\bigl(\tfrac{2jk\pi }{p}\bigr)\Big)_{1\le j,k\le %(p-1)/2}=(-1)^{\tfrac{(p-1)(p-3)}{8}}2^{-\tfrac{p-1}{2}}p^{\tfrac{p-1}{4}}. 
%\end{equation}
Furthermore, the prime case of Conjecture \ref{conj:new5.2} was also settled by  Guo in \cite{g}. Let $p_1=(p-1)/2$ and let $\mathbf T_p=\big(\tan\!(\tfrac{jk\pi}{p})\big)_{1\le j,k\le (p-1)/2}$. It was shown that
\begin{equation}\label{Guotan}
\det(\mathbf T_p)= i^{\,p_1}\,(-1)^{\left\lfloor \frac{p_1}{2}\right\rfloor + p_1}
\left(\prod_{\chi  \, \text{odd}} \bigl(1-2\chi(2)\bigr)\right)
\cdot \prod_{\chi\, \text{odd}} B_{1,\chi}
\cdot \prod_{\chi\, \text{odd}} \overline{\tau(\chi)} \, ,
\end{equation}
 where $B_{1,\chi}$ is the generalized Bernoulli number and $\tau(\chi)$ is the Gauss sum attached to $\chi$ (for precise definitions, see Section 5).
 
In this paper, we will extend Guo’s results (\ref{Guocot})  and (\ref{Guotan}) to arbitrary positive integers $n$ and we will prove Conjecture \ref{conjcsc}. Following Wang’s generalization of Maillet’s determinant, we study the cotangent matrices of the form
$$A_n=(\cot(\frac{jk\pi}{n}))_{j,k\in S_n}.$$

Let $B_n=\bigl(\sin(\tfrac{2jk\pi}{n})\bigr)_{j,k\in S_n}$ and
$M'_n=\bigl(R(ks')-\tfrac{n}{2}\bigr)_{k,s\in S_n}$.

\textbf{Prime case.}
Let $p$ be an odd prime. The key technical ingredient in Guo’s proof is the following
matrix identity for $n=p$,
\begin{equation*}\label{eq:prime-case}
	M'_p=-A_p B_p,\qquad \det(M'_p)=-\tfrac{1}{2}\,D_p.
\end{equation*}

\textbf{Composite case.}
When $n$ is composite, the formula becomes
\begin{equation*}\label{eq:composite-factorization}
		M'_n= -\,\mathbf S_n\,\mathbf C_n,
\end{equation*}
where 
\begin{equation*}
	\mathbf S_n= \bigl(\sin(\tfrac{2jk\pi}{n})\bigr)_{\substack{j\in S_n\\ 1\le k\le n_1}},
\qquad\mathbf C_n= \bigl(\cot(\tfrac{ks\pi}{n})\bigr)_{\substack{1\le k\le n_1\\ s\in S_n}},
\end{equation*}
and $n_1=\bigl\lfloor\tfrac{n-1}{2}\bigr\rfloor.$

 Guo’s approach does not directly extend to our situation. The reason is that the above identity does not yield the desired result due to some technical obstacles. When $n$ is composite, neither of the two matrices on the right-hand side is square. This makes the identity difficult to handle even after applying Cauchy--Binet formula. Moreover, in contrast to the prime case, the determinant of $M'_n$ could not follow immediately from the Maillet determinant $D_n$. This forces us to use a totally different method to obtain the desired result.

Let $L(s,\chi)=\sum\limits_{m=1}^{\infty}\frac{\chi(m)}{m^s}$ be the $L$-function attached to a character $\chi$. Throughout our research, it turns out that our result is intrinsically related to the values of the $L$-functions at $s=1$. Using the spectral decomposition of a suitable linear operator, we establish our main result.
\begin{thm}\label{cotthm}
	Let $n\geq 3$ be a positive integer. Then
	$$|\det(A_n)|=\bigg(\frac{n}{\pi}\bigg)^{\frac{\phi(n)}{2}}\prod_{\chi \text{odd}}L(1,\chi).
	$$
	Equivalently, an explicit form convenient for computation is
	$$|\det(A_n)|=\frac{(2n)^{\frac{\phi(n)}{2}}h_n^-}{Qw\sqrt{\prod\limits_{\chi \text{odd}}f_{\chi}}}\prod_{\chi \text{odd}}\prod\limits_{\substack{p\mid n\\p\nmid f_{\chi}}}\biggl(1-\frac{\chi^*(p)}{p}\biggr),$$
		where  $\chi^*$ stands for the primitive character associated with $\chi$ in the sense of Lemma \ref{char}, $w$ is the number of roots of unity in $\mathbb{Q}(\zeta_n)$, $f_{\chi}$ is the conductor of the character $\chi$,  $h_n^-$ denotes the relative class number of $\mathbb{Q}(\zeta_n)$ and
	$$Q=\begin{cases}
		1&\text{when n is a prime power}  ,\\
		2&\text{otherwise}.\\
	\end{cases}$$
	Moreover, when $n$ is an odd integer, the sign of $\det(A_n)$ is $\varepsilon(n)$ where $\varepsilon(n)$ is as Lemma \ref{sign}.
\end{thm}
As a direct consequence of Theorem \ref{cotthm}, we obtain the corollary as follows.
  \begin{cor}\label{cotcor} Let $n=p^t$ with $p$ an odd prime. Then
  	$$\det(A_{p^t})=\varepsilon(p^t)\frac{(2p^t)^{\tfrac{p^{t-1}(p-1)}{2}-1}h_{p^t}^{-}}{\sqrt{\prod\limits_{\chi \text{odd}}f_{\chi}}}.$$
  	In particular, when $t=1$,
  	$$\det(A_p)=\left(\tfrac{-2}{p}\right)2^{\tfrac{p-3}{2}}p^{\tfrac{p-5}{4}}h_{p}^{-}.$$
  	\end{cor}
  \begin{rem}
  	This corollary recovers Theorem 2.1 of \cite{g}.
  	\end{rem}

Let $$T_n=\bigg(\tan(\tfrac{jk\pi}{n})\bigg)_{1\le j,k\le (n-1)/2}.$$ Motivated by the identity (\ref{Guotan}), we study its minor of order $\phi(n)/2$ of $\mathbf T_{n}$
denoted by $$\mathbf T^1_n=\bigg(\tan(\tfrac{jk\pi}{n})\bigg)_{j,k\in S_n}.$$ We prove the following theorem.
\begin{thm}\label{tanthm}
	Let $n\geq 3$ be a positive  odd integer. Then
	$$\det(\mathbf T^1_n)=\varepsilon(n)\bigg(\frac{n}{\pi}\bigg)^{\frac{\phi(n)}{2}}\prod_{\chi \text{odd}}(1-2\chi(2))L(1,\chi).
	$$
\end{thm}
\begin{rem}
	Theorem \ref{tanthm} generalizes Guo’s identity (\ref{Guotan}). 
	\end{rem}
	
Let $C_n=\bigg(\csc(\tfrac{2jk\pi}{n})\bigg)_{j,k\in S_n}.$ We also obtain a corresponding result for $C_n$.
\begin{thm}\label{cscthm}
	Let $n\geq 3$ be a positive odd integer. Then
		$$\det(C_n)=\varepsilon(n)\bigg(\frac{n}{\pi}\bigg)^{\frac{\phi(n)}{2}}\prod_{\chi \text{odd}}(1-\chi(2))L(1,\chi).
	$$
\end{thm}

The following corollary, as a consequence of Theorem \ref{cscthm}, implies Conjecture \ref{conjcsc}.
\begin{cor}\label{csccor} Let $p$ be an odd prime. Let $\ell$ be the order of $2$ in the multiplicative group $(\mathbb Z/p\mathbb Z)^\times$. Then
	$$\det(C_{p})=	
	\begin{cases}
		\displaystyle \left(\tfrac{-2}{p}\right)2^{\frac{p-1}{\ell}}2^{\tfrac{p-3}{2}}p^{\tfrac{p-5}{4}}h_{p}^{-}, & \text{if $\ell$ is even},\\[6pt]
			\displaystyle	0, & \text{if $\ell$ is odd}.
	\end{cases}$$
	In particular, if $p\equiv 7\bmod{8}$, then $\det(C_{p})=0$.
\end{cor}	
\begin{rem}
	Our results indicate that Sun’s Conjecture \ref{conjcsc} needs a refinement. For example, via {\tt Mathematica} we find that when $p=43$, $c_p=844$. This provides a counterexample. The optimal form is obtained by Corollary \ref{csccor}.
	\end{rem}
	
	We recall a classical identity
	\begin{equation}\label{Guosin}
	\det\!\Big(\sin\!\bigl(\tfrac{2jk\pi }{p}\bigr)\Big)_{1\le j,k\le (p-1)/2}=(-1)^{\tfrac{(p-1)(p-3)}{8}}2^{-\tfrac{p-1}{2}}p^{\tfrac{p-1}{4}}. 
	\end{equation}
	 Inspired by the identity (\ref{Guosin}), we also derive the following theorem.
	\begin{thm}\label{sinthm} Let $n\geq 3$ be a positive integer. Then
		\begin{equation*}
			|\det(B_n)|=\begin{cases}
				2^{-\phi(n)/2}\prod\limits_{\substack{\chi\bmod{n}\\ \chi\text{odd}}} \sqrt{f_{\chi}}\qquad &\text{when  n is odd and square-free},\\
				2^{-\phi(\tfrac{n}{2})/2}\prod\limits_{\substack{\chi\bmod{\tfrac{n}{2}}\\ \chi\text{odd}} }\sqrt{f_{\chi}}\qquad &\text{when  n is even and square-free},\\
				0\qquad &\text{otherwise}.
			\end{cases}
		\end{equation*}
	\end{thm}

	\section{Determinants of cotangent matrices}
	
Before we start to prove the main theorem, we establish some preliminary results. 
\begin{lem}\label{sign}
	Let $n$ be a positive integer and $g$ be an odd fucntion defined on $\mathbb{Z}$ satisfying with period $n$ (i.e. $g(i+n)=g(i)$ for all $i\in\mathbb{Z}$). Define
	\begin{equation*}
		F_n=\big(g(jk)\big)_{j,k\in S_n},\qquad F_n'=\big(g(jk')\big)_{j,k\in S_n}.
	\end{equation*}
	Then 
		$$|\det(F'_n)|=|\det(F_n)|.$$
	Moreover, if $n= \prod\limits_{i=1}^r p_i^{e_i}$ is odd where $p_1,\dots,p_r$ are distinct odd primes, then 
	$$\det(F_n)=\varepsilon(n)\det(F'_n).$$
where $$\varepsilon(n)=
\begin{cases}
	(-1)^{\phi(n)/2}&r=1\ \text{and}\ p_1\equiv1\ \text{or}\ 4e_1+3\pmod8,\\
	(-1)^{\phi(n)/2+1}&r=1\ \text{and}\ p_1\not\equiv1\ \text{and}\ 4e_1+3\pmod8,\\
	(-1)^{\phi(n)/2+1}&r=2\ \text{and}\ p_1+p_2\equiv0\pmod4,\\
	(-1)^{\phi(n)/2}& \text{otherwise}.
\end{cases}$$
In particular, when $n=p$ is an odd prime, $\varepsilon(p)=\left(\tfrac{-2}{p}\right).$
\end{lem}
\begin{proof}
	Let $k\in S_n$ and let  $\tilde k\in S_n$ be the unique integer such that $k\tilde k\equiv \pm1 \pmod{n}$ and define $\tau_n$: $S_n\rightarrow S_n$ by $\tau_n(k)=\tilde k$. Obviously,  $\tau_n$ is a permutation of $S_n$.
	Since $g$ is an odd function, we obtain \(F'_n\) from \(F_n\) in two steps:
	\begin{enumerate}
		\item reorder the columns according to the permutation \(k\mapsto \tilde k\),
		\item multiply by \(-1\) each column indexed by $k$ with \(k\tilde k\equiv -1 \pmod{n}\).
	\end{enumerate}
Hence
\[
\det(F_n')=\operatorname*{sign}(\tau_n(k))(-1)^{\#N_n^{inv}}\det(F_n).
\]
where $N_n^{inv}=\{k\in S_n\,|\,\exists\,\tilde k\in S_n \,\text{such that}\, k\tilde k\equiv -1 \pmod{n}\}.$
In particular,
$$|\det(F'_n)|=|\det(F_n)|.$$
If $n = \prod\limits_{i=1}^r p_i^{e_i}$ is an odd integer where $p_1,\dots,p_r$ are distinct odd primes, then by Theorem 1.1 of \cite{s2}, 
$$\operatorname*{sign}(\tau_n(k))=
\begin{cases}
	-1&r=1\ \text{and}\ p_1\equiv1\ \text{or}\ 4e_1+3\pmod8,\\
	-1&r=2\ \text{and}\ p_1+p_2\equiv0\pmod4,\\
	1& \text{otherwise}.
\end{cases}$$
Moreover, by Lemma 2.1 of \cite{s2},
$$\#N_n^{inv}\equiv \frac{\phi(n)}{2}-\delta_{r,1}\pmod{2},$$
where $\delta_{r,1}$ is $1$ or $0$ according as $r = 1$ or not. Combining these yields，
$$\det(F'_n)=\varepsilon(n)\det(F_n).$$
 In other words, $\det(F_n)=\varepsilon(n)\det(F'_n)$. In particular, when $n=p$ is an odd prime, a direct computation shows that
$$\varepsilon(p)=\operatorname*{sign}(\tau_p(k))(-1)^{\#N_p^{inv}}=\left(\tfrac{-2}{p}\right).$$

	\end{proof}

If $\chi$ is not primitive, the following tools will allow us to handle this case.

	\begin{lem}[\cite{pp}]\label{char}
		Let $\chi $ be any Dirichlet character $\bmod{q}$. Then there exist a unique
		divisor $q^\ast\mid q$ and a unique \emph{primitive} character
		$\chi^\ast $ $\bmod{q^{\ast}}$such that for any $n$ coprime to $q$,
		\begin{equation}\label{eq:agree-on-coprime}\tag{14}
			\chi(n)=\chi^\ast(n).
		\end{equation}
		Conversely, for any divisor $q^\ast\mid q$ and any primitive character
		$\chi^\ast \bmod{q^\ast}$, there exists a unique character $\chi \bmod q$
		for which \eqref{eq:agree-on-coprime} holds whenever $(n,q)=1$. In fact,
		\begin{equation}\label{eq:product-with-principal}\tag{15}
			\chi(n)=\chi^{\ast}(n)\,\chi_{0}(n),
		\end{equation}
		where $\chi_{0}\bmod{q}$ denotes the principal character .
	\end{lem}
	
	\begin{lem}[\cite{yw}]\label{Lfunction}If $\chi$ is a non-principal character$\pmod n$ with conductor $f_{\chi}$ induced by the primitive character $\chi$ in the sense of Lemma \ref{char}, then
		$$L(s,\chi)=L(s,\chi^{*})\prod\limits_{\substack{p\mid n\\p\,\nmid \, f_{\chi}}}\biggl(1-\frac{\chi^{*}(p)}{p^s}\biggr)$$
	\end{lem}
	
We now establish a formula that expresses a cotangent sum in terms of the values of Dirichlet $L$-functions at $s=1$.
\begin{lem}\label{impri}
	For a Dirichlet odd character $\chi\bmod n$ with conductor $f_{\chi}$, we have
$$\sum_{j\in U_n}\chi(j)\cot(\frac{j\pi}{n})=\frac{2n}{\pi}L(1,\chi)=\frac{2n}{\pi}L(1,\chi^{*})\prod\limits_{\substack{p\mid n\\p\nmid f_{\chi}}}\biggl(1-\frac{\chi^{*}(p)}{p}\biggr).$$
	\end{lem}
\begin{proof}
	The first equality follows from \cite[Theorem 8.2]{kt} and the second can be deduced by Lemma \ref{Lfunction}.
	\end{proof}

We define the function spaces
\begin{align*}
	V_1&=\{\,h \mid h:U_n\to \mathbb{C}\ \text{is a function}\,\},V_2=\{\,h \mid h:S_n\to \mathbb{C}\ \text{is a function}\,\},\\
	V_1^{\pm}&=\{\,h\in V_1 \mid h(-x)=\pm h(x)\,\}.
	\end{align*}
In particular, $V_1^-$ admits a natural identification with $V_2$ as stated in the next lemma.
\begin{lem}\label{id}
The space $V_1^-$ is canonically isomorphic to the function space $V_2$.
\end{lem}
\begin{proof}
Given $g\in V_1^-$, its restriction to $S_n$ determines $g$ uniquely, because if $k\in S_n$, then $-k$ lies outside $S_n$ and $g(-k)=-g(k)$. Thus the map
\[
\Phi:V_1^- \to  V_2, \quad f \mapsto f|_{S_n}
\]
is injective. It is also surjective: given any function $h:S_n\to\mathbb C$, we may extend it to $U_n$ by setting $g(k)=h(k)$ for $k\in S_n$, $g(-k)=-h(k)$. This $g$ belongs to $V_1^-$ and $\Phi(g)=h$. Hence $\Phi$ is a linear isomorphism. 
\end{proof}
Next, we recall a well-known result with respect to the values of $L$-functions at $s=1$.
	\begin{prop}[\cite{w2}]\label{h}
	Let $\{\chi_1,\ldots,\chi_m\}$ be the set of all odd Dirichlet characters
	modulo $n$ and $K=\mathbb{Q}(\zeta_n)$. Suppose $K^+$ be its maximal real subfield, then 
	$$\prod_{\chi \thinspace\text{odd}}L(1,\chi^*)=\frac{(2\pi)^{\frac{\phi(n)}{2}}h_n^-}{Qw\sqrt{\mid d(K)/d(K^+)\mid}}=\frac{(2\pi)^{\frac{\phi(n)}{2}}h_n^-}{Qw\sqrt{\prod\limits_{\chi \text{odd}}f_{\chi}}},
	$$
	where $\chi^*$ is primitive character associated with $\chi$ and $f_{\chi}$, $Q$, $w$, $h_n^-$ are as in Theorem \ref{cotthm}.
\end{prop}

Now, we are in a position to present the proof of Theorem \ref{cotthm}.
\begin{proof}[Proof of Theorem \ref{cotthm}]
		Let $A'_n=\bigg(\cot(\tfrac{jk'\pi}{n})\bigg)_{j,k\in S_n}$. By Lemma \ref{sign},
	$$|\det(A_n)|=|\det(A'_n)|.$$
	and when $n$ is odd, $\det(A_n)=\varepsilon(n)\det(A'_n)$.
	
	For $x \in U_n$, we consider the kernel function
		$$K(x):=\frac{1}{2}\cot(\tfrac{x\pi}{n})$$
		Define the involution operator $T$ on $V_1$ by
		$$(Tg)(j)=\sum_{k \in U_n}K(jk')g(k)=\frac{1}{2}\sum_{k \in U_n}\cot(\tfrac{jk'\pi}{n})g(k),$$
		for any $g\in V_1$. A straightforward computation shows that $T$ annihilates even functions, hence $T$ can be defined on $V_1^{-}$.
		When $g$ is an odd function, 
		$$(Tg)(j)=\sum_{k \in S_n}\cot(\tfrac{jk'\pi}{n})g(k).$$

		By Lemma \ref{id}, define a new operator $T'$ on $V_2$ by $T'(g):=T(\Phi^{-1}(g))$.
	
		Consider the indicator functions on $S_n$:
		\[
		\delta_k(j) =
		\begin{cases}
			1, & j=k,\\
			0, & j\neq k,
		\end{cases}
		\qquad k\in S_n.
		\]
		The family $\{\delta_k : k\in S_n\}$ forms a basis of $V_2$.
		
		With respect to this basis, the operator $T^{'}$ has the matrix
		\[
		A_n' =\Bigl( \cot\!\left(\tfrac{\pi j k'}{n}\right) \Bigr)_{j,k\in S_n}.
		\]
		We now compute the eigenvalues of $T'$. For any odd Dirichlet character $\chi$ and any $j\in S_n$, we obtain by Lemma \ref{impri}
		\begin{align*}
			T'(\chi|_{S_n})=(T\chi)(j) 
			&= \frac{1}{2}\sum_{k \in U_n}\cot(\tfrac{jk'\pi}{n})\chi(k) \\[6pt]
			&= \frac{1}{2}\sum_{k \in U_n}\cot(\tfrac{k\pi}{n})\chi(jk') \\[6pt]
			&= \frac{\chi(j)}{2}\sum_{k \in U_n}\cot(\tfrac{k\pi}{n})\overline{\chi}(k) \\[6pt]
			&= \chi(j) \frac{n}{\pi}L(1,\overline{\chi}).\\[6pt]
		\end{align*}
Thus for any odd character $\chi$, $\frac{n}{\pi}L(1,\overline{\chi})$ is an eigenvalue of $T'$. As $\chi$ ranges over all odd characters modulo $n$, it follows that the number of eigenvalues is $\frac{\phi(n)}{2}$. There is a natural one-to-one correspondence between each character $\chi$ and its conjugate character $\overline{\chi}$. From this correspondence, we deduce that 
	$$|\det(A_n)|=\det(A'_n)=\bigg(\frac{n}{\pi}\bigg)^{\frac{\phi(n)}{2}}\prod_{\chi \text{odd}}L(1,\chi).$$
By Proposition \ref{h} and Lemma \ref{Lfunction}, we obtain the following explicit formula for $|\det(A_n)|$,
$$|\det(A_n)|=\det(A'_n)= \frac{(2n)^{\frac{\phi(n)}{2}}h_n^-}{Qw\sqrt{\prod_{\chi \text{odd}}f_{\chi}}}\prod\limits_{\substack{p\mid n\\p\nmid f}}\biggl(1-\frac{\chi^*(p)}{p}\biggr),$$
By the above expression, $\det(A'_n)>0$. If $n$ is odd, then $\operatorname*{sign}(\det(A_n))=\varepsilon(n)$.
		\end{proof}
Using Wang’s method from the proof of (\ref{wangk}), we provide an alternative proof of  Theorem \ref{cotthm}. 		
	\begin{proof}[Alternative Proof of Theorem \ref{cotthm}]
		First of all, we compute the determinants of $A'_n$. Let $\{\chi_1,\ldots,\chi_m\}$ be the set of all odd Dirichlet characters
		modulo $n$, where $m=\phi(n)/2$. Let $S_n=\{a_1,\ldots,a_m\}$ and set
		\[
		\Omega=m^{-1/2}\,[\chi_i(a_j)]_{1\le i,j\le m}.
		\]
		By the orthogonality relations for Dirichlet characters,
		\[
		\frac{1}{\phi(n)}\sum_{a\in(\mathbb{Z}/n\mathbb{Z})^{\times}}\chi_i(a)\,\overline{\chi_j(a)}=\delta_{ij},
		\]
		so $\Omega$ is a unitary matrix, and
		\[
		\sum_{{\substack{\chi\bmod{n}\\ \chi\text{odd}}}}\chi(a)=
		\begin{cases}
			m,& a=1,\\
			-m,& a=n-1,\\
			0,& \text{otherwise},
		\end{cases}
		\qquad
		\sum_{a\in S_n}\chi_i(a)\,\overline{\chi_j(a)}=m\,\delta_{ij}.
		\]
		Note that
		\[
		\det\!\Bigl(\cot\!\bigl(\tfrac{ab'\pi}{n}\bigr)\Bigr)_{a,b\in S_n}
		=\det\!\Bigl(\Omega\;(\cot\!\bigl(\tfrac{ab'\pi}{n}\bigr))_{a,b\in S_n}\;\Omega^\ast\Bigr)
		=\det\!\Bigl(\tfrac{1}{m}\,(s_{ij})_{1\le i,j\le m}\Bigr),
		\]
		where
		\[
		s_{ij}:=\sum_{a\in S_n}\sum_{b\in S_n}\chi_i(a)\,
		\cot\!\left(\tfrac{ab'\pi}{n}\right)\overline{\chi_j(b)} .
		\]
		
	One can show that
		\begin{align*}
			&\sum_{a\in(\mathbb Z/n\mathbb Z)^\times}\ \sum_{b\in(\mathbb Z/n\mathbb Z)^\times}
			\chi_i(a)\,\cot(\frac{ab'\pi}{n})\,\overline{\chi_j(b)}\\
			&=\sum_{c\in(\mathbb Z/n\mathbb Z)^\times}\ \sum_{b\in(\mathbb Z/n\mathbb Z)^\times}
			\chi_i(cb)\,\cot(\frac{c\pi}{n})\,\overline{\chi_j(b)}\\
			&=\left(\sum_{b\in(\mathbb Z/n\mathbb Z)^\times}\chi_i(b)\,\overline{\chi_j(b)}\right)
			\left(\sum_{c\in(\mathbb Z/n\mathbb Z)^\times}\cot(\frac{c\pi}{n})\,\chi_i(c)\right)\\
			&=\delta_{ij}\,\frac{2n\phi(n)}{\pi}\,L(1,\chi).
		\end{align*}
		
		On the other hand,
		\begin{align*}
			\sum_{a\notin S_n}\ \sum_{b\in S_n}\chi_i(a)\,\cot(\frac{ab'\pi}{n})\,\overline{\chi_j(b)}
			&=\sum_{a\in S_n}\ \sum_{b\in S_n}\chi_i(-a)\,\cot(\frac{-ab'\pi}{n})\,\overline{\chi_j(b)}\\
			&=s_{ij}.
		\end{align*}
		Similarly, we have
		\begin{align*}
			\sum_{a\in S_n}\ \sum_{b\notin S_n}\chi_i(a)\,\cot(\frac{ab'\pi}{n})\,\overline{\chi_j(b)}
			&=s_{ij},\\
			\sum_{a\notin S_n}\ \sum_{b\notin S_n}\chi_i(a)\,\cot(\frac{ab'\pi}{n})\,\overline{\chi_j(b)}&=s_{ij}.
		\end{align*}
		Hence
		$$s_{ij}=\delta_{ij}\,\frac{n\phi(n)}{2\pi}\,L(1,\chi)$$
		Then
		$$\det(A'_n)=\bigg(\frac{n}{\pi}\bigg)^m\prod_{\chi \text{odd}}L(1,\chi).
		$$
	The rest proceeds exactly as in the previous proof, so we omit the details.
	\end{proof}

Meanwhile, we also give the proof of Corollary \ref{cotcor}.
	\begin{proof}[Proof of Corollary \ref{cotcor}]

	When $n=p^t$, there is no prime $p$ with $p\mid n$ and $p\nmid f$. Consequently,
	$$\prod\limits_{\substack{p\mid n\\p\nmid f}}\biggl(1-\frac{\chi^{\ast}(p)}{p}\biggr)=1.$$
In this case we obtain $Q=1$ and $w=2p^t$. Then by Theorem \ref{cotthm},
		$$\det(A_{p^t})=\varepsilon(p^t)\frac{(2p^t)^{\tfrac{p^{t-1}(p-1)}{2}-1}h_{p^t}^{-}}{\sqrt{\prod_{\chi \text{odd}}f_{\chi}}}.
	$$
In particular, when t=1,
$$|\det(A_{p})|=\varepsilon(p)2^{\tfrac{p-3}{2}}p^{\tfrac{p-5}{4}}h_{p}^{-}.$$
Then by Lemma \ref{sign},
$$\det(A_{p})=\left(\tfrac{-2}{p}\right)2^{\tfrac{p-3}{2}}p^{\tfrac{p-5}{4}}h_{p}^{-}.$$
\end{proof}

\section{Determinants of tangent matrices}
\begin{lem}\label{tanimpri}
	Let $n$ be a positive odd integer. For a Dirichlet odd character $\chi\bmod n$, we have
	$$\sum_{j\in U_n}\chi(j)\tan(\frac{j\pi}{n})=(1-2\overline{\chi}(2))\frac{2n}{\pi}L(1,\chi).$$
\end{lem}
\begin{proof}
	Since
	$$\tan(x)=\cot(x)-2\cot(2x),$$
	it follows that from Lemma \refeq{impri},
	\begin{align*}
		\sum_{j\in U_n}\chi(j)\tan(\frac{j\pi}{n})&=\sum_{j\in U_n}\chi(j)\cot(\frac{j\pi}{n})-
		2\sum_{j\in U_n}\chi(j)\cot(\frac{2j\pi}{n})\\
		&=(1-2\overline{\chi}(2))\sum_{j\in U_n}\chi(j)\cot(\frac{j\pi}{n})\\
		&=(1-2\overline{\chi}(2))\frac{2n}{\pi}L(1,\chi).
	\end{align*}
\end{proof}
We are now giving the proof of Theorem \ref{tanthm}.
\begin{proof}[Proof of Theorem \ref{tanthm}]
	Let $\mathbf T^{(2)}_n=\bigg(\tan(\frac{jk'\pi}{n})\bigg)_{j,k\in S_n}$. By Lemma \ref{sign},
	$$|\det(\mathbf T^1_n)|=|\det(\mathbf T^{(2)}_n)|.$$
		Define the operator $T_1$ on $V_1$ by
	$$(T_1g)(j)=\frac{1}{2}\sum_{k \in U_n}\tan(\frac{jk'\pi}{n})g(k),$$
	for any $g\in V_1$. Repeating process as in the proof of Theorem \ref{cotthm}, we have
	
	$$\det(\mathbf T^{(2)}_n)=(\frac{1}{2})^{\phi(n)/2}\prod\limits_{\chi \text{odd}}\sum_{j\in U_n}\chi(j)\tan(\frac{j\pi}{n}).$$
	Then, by Lemma \ref{tanimpri},
	$$|\det(\mathbf T^1_n)|=\det(\mathbf T^{(2)}_n)=\bigg(\frac{n}{\pi}\bigg)^{\frac{\phi(n)}{2}}\prod_{\chi \text{odd}}(1-2\chi(2))L(1,\chi).$$
If $n$ is odd, it follows that
	$$\det(\mathbf T^1_n)=\varepsilon(n)\bigg(\frac{n}{\pi}\bigg)^{\frac{\phi(n)}{2}}\prod_{\chi \text{odd}}(1-2\chi(2))L(1,\chi).$$
\end{proof}

\section{Determinants of cosecant matrices }
\begin{lem}\label{cscimpri}
	Let $n$ be a positive odd integer. For a Dirichlet odd character $\chi\bmod n$, we have
	$$\sum_{j\in U_n}\chi(j)\csc(\frac{2j\pi}{n})=(1-\overline{\chi}(2))\frac{2n}{\pi}L(1,\chi).$$
\end{lem}
\begin{proof}
	Since
	$$\csc(2x)=\frac{1}{2}(\cot(x)+\tan(x)),$$
	it follows that from Lemma \refeq{impri},
	\begin{align*}
		\sum_{j\in U_n}\chi(j)\csc(\frac{2j\pi}{n})&=\frac{1}{2}(\sum_{j\in U_n}\chi(j)\cot(\frac{j\pi}{n})+
		\sum_{j\in U_n}\chi(j)\tan(\frac{j\pi}{n}))\\
		&=(1-\overline{\chi}(2))\sum_{j\in U_n}\chi(j)\cot(\frac{j\pi}{n})\\
		&=(1-\overline{\chi}(2))\frac{2n}{\pi}L(1,\chi).
	\end{align*}
\end{proof}
\begin{lem}\label{lem3.3}
	Let $\ell$ be as Corollary \ref{csccor}. Then
	\[
	\prod_{\chi \text{odd}}\bigl(1-\chi(2)\bigr)
	=
	\begin{cases}
		\displaystyle 2^{\frac{p-1}{\ell}}, & \text{if $\ell$ is even},\\[6pt]
		\displaystyle	0, & \text{if $\ell$ is odd}.
	\end{cases}
	\]
	In particular, $p\equiv 7\bmod{8}$, 
	$$\prod_{\chi \text{odd}}\bigl(1-\chi(2)\bigr)=0.$$
\end{lem}
\begin{proof}
	If $\ell$ is odd, then $\ell\,\mid\,(p-1)/2$ and
	\[
	\prod_{\chi\, \text{odd}}\bigl(1-\chi(2)\bigr)
	=(1-1)^{\frac{p-1}{2\ell}}\,\prod_{i=1}^{l-1}(1-\zeta_\ell^i)^{\frac{p-1}{2\ell}}
	=0,
	\]
	where $\zeta_\ell$ is a primitive $\ell$-th root of unity.

	If $\ell$ is even, then $\ell\,\mid\,p-1$. Hence
	\[
	\prod_{\chi\, \text{odd}}\bigl(1-\chi(2)\bigr)
	=\prod_{i=1}^{l/2}(1-\zeta_\ell^{2i-1})^{\frac{p-1}{\ell}}=2^{\frac{p-1}{\ell}}.
	\]
	In particular, if $p\equiv 7\bmod{8}$, then by quadratic reciprocity $(\frac{2}{p})=1$. Hence $2$ is not a primitive root of $p$. Note that $l\,|\,\frac{p-1}{2}$ and $\frac{p-1}{2}\equiv 3\bmod{4}$, so $\ell$ is odd. It follows that $\prod_{\chi \text{odd}}\bigl(1-\chi(2)\bigr)=0.$
\end{proof}

Now we are turning to the proof of Theorem \ref{cscthm}.

\begin{proof}[Proof of Theorem \ref{cscthm}]
	Let $C'_n=\bigg(\csc(\frac{2jk'\pi}{n})\bigg)_{j,k\in S_n}$. By Lemma \ref{sign},
	$$|\det(C_n)|=|\det(C'_n)|.$$
	and when $n$ is odd, $\det(C_n)=\varepsilon(n)\det(C'_n)$.
	
	Define the operator $T_2$ on $V_1$
	\medskip
	$$(T_2 g)(j)=\frac{1}{2}\sum_{k \in U_n} \csc \left(\frac{2 \pi j k'}{n}\right) g(k).$$
	for any $g\in V_1$.
	Repeating the proof process of Theorem \ref{cotthm}, we have
	
	$$\det(C'_n)=(\frac{1}{2})^{\phi(n)/2}\prod\limits_{\chi \text{odd}}\sum_{j\in U_n}\chi(j)\csc(\frac{2j\pi}{n}).$$
	Then, by Lemma \ref{cscimpri},
	$$|\det(C_n)|=\det(C'_n)=\bigg(\frac{n}{\pi}\bigg)^{\frac{\phi(n)}{2}}\prod_{\chi \text{odd}}(1-\chi(2))L(1,\chi)$$
For odd $n$, we have
	$$\det(C_n)=\varepsilon(n)\bigg(\frac{n}{\pi}\bigg)^{\frac{\phi(n)}{2}}\prod_{\chi \text{odd}}(1-\chi(2))L(1,\chi)$$
\end{proof}
Meanwhile, we also prove Corollary \ref{csccor}.
\begin{proof}[Proof of Corollary \ref{csccor}]
	Similar to the proof of Corollary \ref{cotcor}, the result now follows from Theorem \ref{cscthm}, Proposition \ref{h} and Lemma \ref{sign}.
\end{proof}

\section{Determinants of sine matrices }

Let $\zeta_n$ be the $n$-th root of unity. Recall that $G(l, \chi)=\sum\limits_{j=1}^{n}\chi(j)\zeta_n^j$ is the Gauss sum attached to $\chi$. In particular, for $l=1$ we write $G(1,\chi)=\tau(\chi)$. The Gauss sum satisfies the following well-known properties,
\begin{lem}[\cite{a}]\label{Gauss}
	Let $\chi$ be a Dirichlet character modulo $n$. Then for any $l$ coprime to $n$,
$$
G(l, \chi)=\bar{\chi}(l) \tau(\chi)
$$
\end{lem}

\begin{lem}[\cite{w2}]\label{modulo}
$$|\tau(\chi)|=\sqrt{f_{\chi}}$$
where $f_{\chi}$ is the conductor of $\chi$.
\end{lem}

\begin{lem}[\cite{pp}]\label{de}
	Suppose that $\chi\bmod{n}$ and $\chi^*\bmod{n^*}$ are equivalent in the sense of Lemma \ref{char}, then 
$$\tau\left(\chi\right)=\chi^*\left(\frac{n}{n^*}\right) \mu\left(\frac{n}{n^*}\right) \tau\left(\chi^*\right),$$
where $\mu(n)$ is the M$\ddot{o}$bius function.
\end{lem}

\begin{lem}\label{zero} When $4\nmid n$  and $n$ is not square-free,
	$$\prod_{\chi \text{odd}}\tau(\chi)=0$$
\end{lem}
\begin{proof}
	If $n$ satisfies the above condition, then there exists an odd prime $p$ such that $p^2|n$. Let $\psi$ be a primitive odd Dirichlet character modulo $p$. By Lemma \ref{char}, there exists an odd Dirichlet character $\chi \pmod n$
	induced by $\psi$ such that
	$$
	\chi(m)=\psi(m)\qquad\text{for all }(m,n)=1 .
	$$
	Then by Lemma \ref{de},
\begin{equation*}
	\tau\left(\chi\right)=\psi\left(n/p\right) \mu\left(n/p\right) \tau\left(\psi\right)
\end{equation*}
As $p^2\mid n$, it follows that $(n/p,p)>1$, and thus $\psi(n/p)=0$. Consequently, $\tau\left(\chi\right)=0$ for some
odd Dirichlet character $\chi$. We deduce that
$$\prod_{\chi \text{odd}}\tau(\chi)=0.$$
	\end{proof}
To begin wtih, we consider the determinants of the sine matrices when $n$ is even,
\begin{prop}\label{sineven}
	 Let $n$ be a positive even integer $\geq 3$. Then
	\begin{equation*}
		|\det(B_n)|=\begin{cases}
		|\det(B_{n_1})|&n=2n_1,\enspace n_1\text{odd}\\
			1&n=4\\
			0\qquad &4|n
		\end{cases}
	\end{equation*}
\end{prop}
\begin{proof}
	When $n=4$, a direct computation shows that $\det(B_n)=1$.
	  
	Now assume that $n>4$ with $4\mid n$. For any $j\in S_n$, we note that $j$ must be odd, and moreover $\tfrac{n}{2}-j \in S_n$. For such a pair $(j, \tfrac{n}{2}-j)$ and for any $k\in S_n$, we have
	$$\sin(\frac{2(\frac{n}{2}-j)k\pi}{n})=\sin(\frac{2jk\pi}{n}).$$
	So the matrix has two identical rows and $\det(B_n)=0.$
	
	Now we consider the case when $n=2n_1$, $\sin(\frac{2jk\pi}{n})=\sin(\frac{jk\pi}{n_1})$. Since $(2,n_1)=1$, there exists a multiplicative inverse of $2$ modulo $n_1$. The mapping $k\mapsto 2k$ is a permutation of $S_{n_1}$ possibly with sign changes. Hence the columns of $B_{n_1}$ are merely a reordering of those of $B_n$. In conclusion,
 $$|\det(B_n)|=|\det(B_{n_1})|.$$
\end{proof}
Motived by Proposition \ref{sineven}, we need only to consider the case when $n$ is odd. Then we have the following result,
\begin{thm}\label{sinodd} Let $n$ be a positive odd integer $\geq 3$. Then
	\begin{equation*}
	|\det(B_n)|=\begin{cases}
	2^{-\phi(n)/2}\prod\limits_{\chi \,\text{odd}} \sqrt{f_{\chi}}\qquad &\text{when  n is square-free}\\
	0\qquad &\text{otherwise}.
	\end{cases}
\end{equation*}
	\end{thm}

\begin{proof}
	\noindent
	Let $B'_n=\bigg(\sin(\frac{2jk'\pi}{n})\bigg)_{j,k\in S_n}$. By Lemma \ref{sign},
	$$|\det(B_n)|=|\det(B'_n)|.$$
Define the operator $T_3$ on $V_1$
 \medskip
	$$(T_3 g)(j)=\frac{1}{2}\sum_{k \in U_n} \sin \left(\frac{2 \pi j k'}{n}\right) g(k).$$
	for any $g\in V_1$. We easily see that when $g$ is an odd function,
	$$(T_3 g)(j)=\sum_{k \in S_n} \sin \left(\frac{2 \pi j k'}{n}\right) g(k).$$
	Similar to the argument used in the proof of Theorem \ref{cotthm}, we can find that  $T_3$ can be defined on $V_1^{-}$
and obtain an new operator $T_3'(g):=T(\Phi^{-1}(g))$ on $V_2$. With respect to this basis $\{\delta_k : k\in S_n\}$, the operator $T_3^{'}$ has matrix
	\[
	B'_n =\Bigl( \sin\!\left(\tfrac{2\pi j k'}{n}\right) \Bigr)_{j,k\in S_n}.
	\]
We now compute the eigenvalues of $T_3'$. For an odd character $\chi$ and $j\in S_n$,
	\begin{align*}
		T_3'(\chi|_{S_n})=(T_3\chi)(j) 
		&= \frac{1}{2}\sum_{k\in U_n}\bigl(\sin(\frac{2\pi jk'}{n}\bigr)\chi(k) \\[6pt]
		&= \frac{1}{2}\sum_{k\in U_n}\bigl(\sin(\frac{2\pi jk}{n}\bigr)\chi(k') \\[6pt]
		&= \frac{1}{4i}\sum_{k\in U_n}\bigl(\zeta_n^{jk}-\zeta_n^{-jk}\bigr)\chi(k') \\[6pt]
		&= \frac{1}{4i}\biggl(\sum_{k\in U_n}\zeta_n^{jk}\chi(k')- \sum_{k\in U_n}\zeta_n^{jk}\chi(-k')\biggr) \\[6pt]
		&=  \frac{\chi(1)-\chi(-1)}{4i}\sum_{k\in U_n}\zeta_n^{jk}\chi(k')\\[6pt]
		&= \frac{1}{2i}\chi(j)\tau(\overline{\chi}),
	\end{align*}
where the final step follows from Lemma \ref{Gauss}.
 
  Thus $\tau(\overline{\chi})/2i$ is an eigenvalue of $T_3'$ and as $\chi$ ranges over all odd characters of $n$, the number of $\tau(\overline{\chi})/2i$ is $\frac{\phi(n)}{2}.$
  Together with the pairing \(\chi\leftrightarrow\overline\chi\), we can deduce that 
  $$|\det(B_n)|=\bigg(\frac{1}{2i}\bigg)^{m}\prod_{\chi \text{odd}}\mid\tau(\chi)\mid.$$
  By Lemma \ref{zero}, $|\det(B_n)|=0$ when $n$ is not square-free. Moreover, by Lemma \ref{modulo}, when $n$ is square-free,
  $$|\det(B_n)|=2^{-m}\prod\limits_{\chi \,\text{odd}} \sqrt{f_{\chi}},$$
which completes the proof of Theorem.
\end{proof}
We  now turn to the proof of Theorem \ref{sinthm}.
\begin{proof}[Proof of Theorem \ref{sinthm}]
	The result now follows from Proposition \ref{sineven} and Theorem \ref{sinodd}.
	\end{proof}

\subsubsection*{\noindent\textbf{Funding}} {\small The authors are supported by National Nature Science Foundation of China (Nos. 11971226, 12231009).}
\subsubsection*{\noindent\textbf{Data Availablity}} {\small No addtional data is available.}

\subsection*{\normalfont\Large\bfseries Declarations}
\par
\subsubsection*{\noindent\textbf{\small Conflict of interest:}} {\small The author has no relevant financial or non-financial interests to disclose.}

\end{document}